\documentclass[11pt,reqno]{amsart}
\usepackage{amssymb}
\usepackage{amsfonts}
\usepackage{mathrsfs}
\usepackage{amsmath}
\usepackage{graphicx}
\usepackage{hyperref}
\usepackage{float}
\usepackage{epstopdf}
\usepackage{color}
\usepackage{bm}
\usepackage{comment}
\usepackage{soul}
\usepackage{esint}

\allowdisplaybreaks
\newtheorem{theorem}{Theorem}[section]

\newtheorem{lemma}[theorem]{Lemma}

\newtheorem{remark}[theorem]{Remark}

\newtheorem{definition}[theorem]{Definition}

\def\mcE{\mathcal{E}}
\def\mcF{\mathcal{F}}
\def\msG{\mathscr{G}}

\numberwithin{equation}{section}

\begin{document}
\title[A Sierpinski carpet like fractal without standard self-similar energy]{A Sierpinski carpet like fractal without standard self-similar energy}

\author{Shiping Cao}
\address{Department of Mathematics, Cornell University, Ithaca 14853, USA}
\email{sc2873@cornell.edu}
\thanks{}

\author{Hua Qiu}
\address{Department of Mathematics, Nanjing University, Nanjing, 210093, P. R. China.}
\thanks{The research of Qiu was supported by the National Natural Science Foundation of China, grant 12071213, and the Natural Science Foundation of Jiangsu Province in China, grant BK20211142.}
\email{huaqiu@nju.edu.cn}

\subjclass[2010]{Primary 28A80, 31E05.}

\date{}

\keywords{Sierpinski carpet like fractals, Dirichlet forms, sub-Gaussian heat kernel estimate, resistance metric}

\maketitle

\begin{abstract}
We construct a Sierpinski carpet like fractal, on which a self-similar diffusion with sub-Gaussian heat kernel estimate does not exist, in contrast to previous researches on the existence of such diffusions, on the generalized Sierpinski carpets and recently introduced unconstrained Sierpinski carpets.  
\end{abstract}

\section{Introduction}\label{sec1}
In history, the existence of locally symmetric diffusions with sub-Gaussian heat kernel estimates on Sierpinski carpets (SC) was proved in the pioneering works \cite{BB,BB1,BB2} by Barlow and Bass, using a probability method. One key step in the proof is to show certain paths, named `Knight moves' and `corner moves', occur with positive probability, and the argument relies on the local symmetry of the fractals. By introducing the difficult coupling argument, the result was later extended to generalized Sierpinski carpets (GSC) \cite{BB3}, which are higher dimensional analogues of SC.

A different approach, later shown in \cite {BBKT} to be equivalent with Barlow and Bass's \cite{BB,BB3}, was introduced by Kusuoka and Zhou \cite{KZ}. The strategy is to construct  self-similar Dirichlet forms on fractals as limits of averaged rescaled energies on cell graphs. This approach is analytic, except a key step to verify that the resistance constants and the Poincare constants are comparable. In particular, the probabilistic `Knight move' argument was borrowed to achieve this, see Theorem 7.16 in \cite{KZ}.  This gap was fulfilled in the recent work \cite{CQ} of the authors by a chaining argument of resistances, and the existence theorem extends to a class of carpet-like fractals, named unconstrained Sierpinski carpets (USC).  In some sense, the USC are more flexible in geometry as  cells except those along the boundary are allowed to live off the grids. See Figure \ref{fig1} for examples of SC and USC. An analytic approach on GSC when the Hausdorff dimension is greater than $2$ still remains unknown.

\begin{figure}[htp]
	\includegraphics[width=4.9cm]{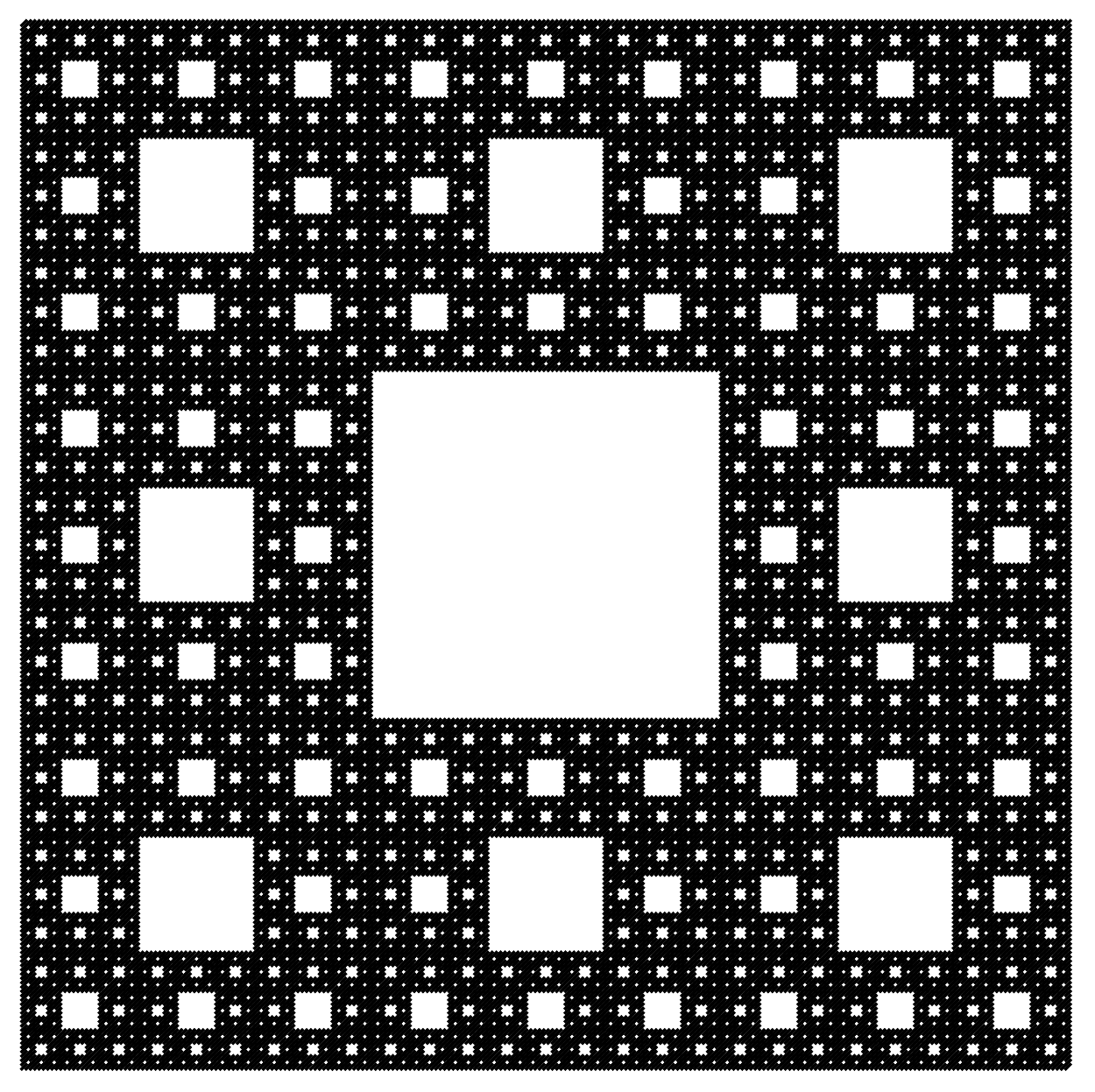}\hspace{1cm}
	\includegraphics[width=4.9cm]{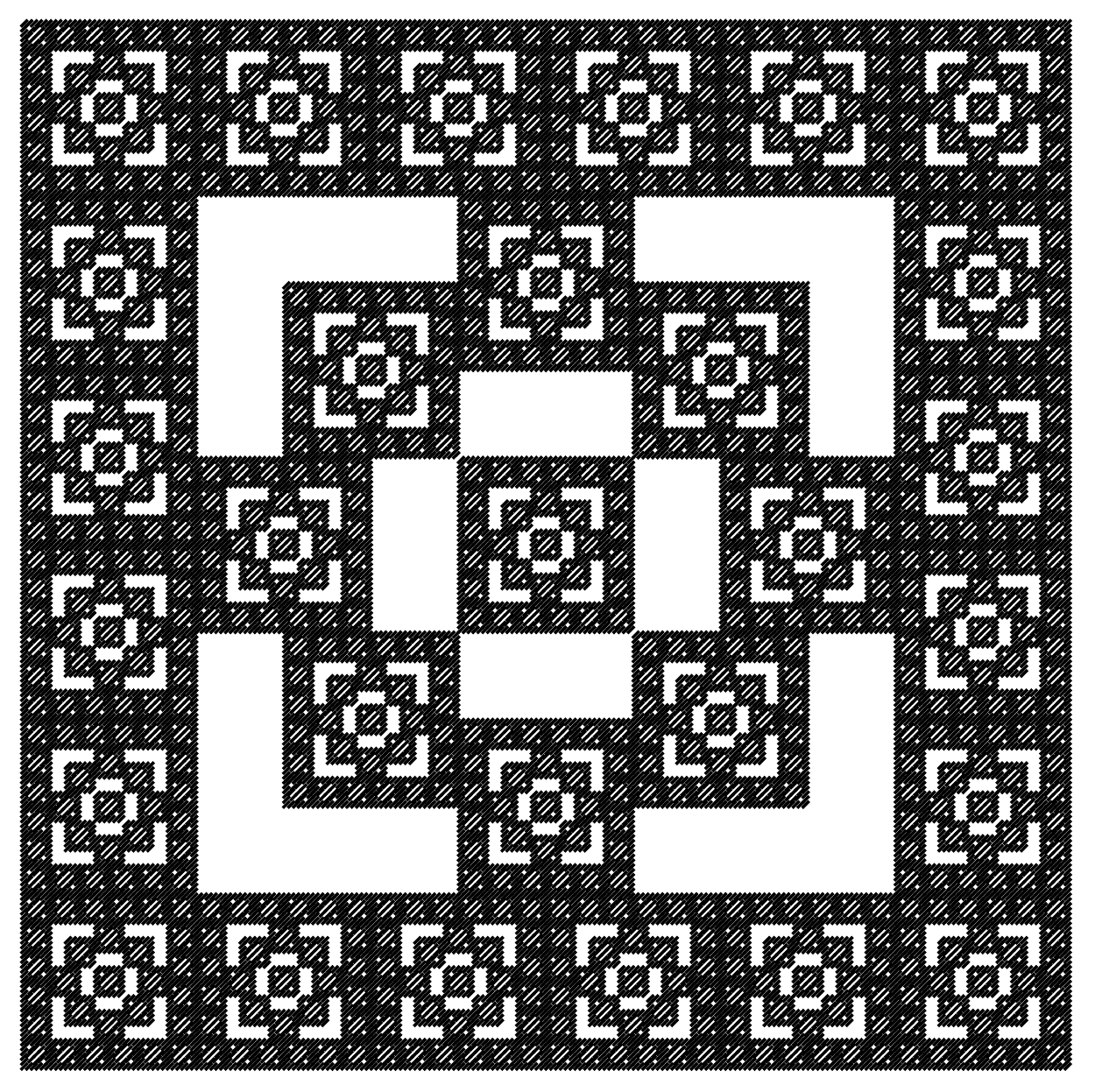}
	\caption{Two examples of USC (the left one is the standard SC).}
	\label{fig1}
\end{figure}

It is believed by the authors that the same approach may be adapted to some planar fractals with relaxed boundary restriction, or even with  contraction ratios of $1$-cells allowed to be distinct, but still keeping the high symmetry. This naturally leads us to look at a larger family, named as  Sierpinski carpet like fractals (LSC), see  Definition \ref{def21}. Unexpectedly, in this paper, we will present an example of LSC, which does not have a `nice' self-similar Dirichlet form. Precisely, there does not exist a self-similar diffusion with sub-Gaussian heat kernel estimate on this fractal. In addition, this example can also be easily generalized to the higher dimensional case.

Before proceeding, we briefly explain the analytic approach to the existence  \cite{KZ,CQ}. This will help readers to see the intuition that motivate the construction (see Remark \ref{rem35}). In \cite{KZ}, for a large class of fractals, four sequences of constants were introduced, and the two most important ones are the \textit{Poincare constants} $\lambda_n$ and the \textit{resistance constants} $R_n$ (we adapt $\lambda_n$ here by a scaling to help readers to see the relation): 
\[\lambda_n=\sup\limits _{f\in l^2(W_n)}\frac{[(f-[f]_{W_n})^2]_{W_n}}{E_n(f)}, \]  
where $W_n$ denotes the set of $n$-cells of the fractal, $[f]_{W_n}$ denotes the average of $f$ on $W_n$, and $E_n$ is the natural cell graph energy on $W_n$;
\[R_n=\inf\limits_{w\in W_m,m\geq 1}R_{m+n}(w\cdot W_n,C_w\cdot W_n),\]
where $R_{m+n}(X,Y)$ is the effective resistance associated with the energy form $E_{m+n}$ between two subsets $X,Y$ in $W_{m+n}$, $w\cdot W_n$ is the set of $(m+n)$-cells contained in $w$, $C_w$ consists of those $m$-cells that are not neighboring to $w$, and $C_w\cdot W_n$ is the set of $(m+n)$-cells contained in $C_w$. It is natural to see $\lambda_n\gtrsim  R_n$, and moreover for the recurrent case in \cite{KZ} (also see \cite{CQ} for the reorganized version where the constants $\lambda^D_n$ and some related estimates are dropped), it is proved that
\[
R_n\lambda_m\lesssim \lambda_{n+m}\lesssim \lambda_n\lambda_m.
\]
To achieve an exponential growth order of $R_n$, we need a remaining piece of estimate $\lambda_n\lesssim R_n$. This was not proved analytically until recently in \cite{CQ}. The proof is based on the fact that $\lambda_n\geq \rho^m\lambda_{n-m}$ for some $\rho>1$ (for the lower dimensional case), so $\lambda_n\gg \lambda_{n-m'}$ if we fix some $m'$ large enough. In particular, if we choose $f$ such that $E_n(f)=1$ and $[f^2]_{W_n}-[f]_{W_n}^2=\lambda_n$, then the variation of $f$ on each $(n-m')$-cell (and also neighboring $n$-cells) is negligible with respect to the total variation of $f$. Using this convenience, on a SC, one can pick two $(n-m')$-cells with difference of $f$ comparable with $\lambda_{n-m'}$, and thus the resistance between opposite boundary lines can be estimated from below in terms of $\lambda_n$. On USC, based on the same idea, a more complicated argument can be applied to find a chain of $(n-2m')$-cells, lying along the boundary of some $(n-m')$-cell. In addition, a linearization extension argument is provided to fulfill  the desired resistance estimate. See Section 4 in \cite{CQ} for details. 

The contents of this paper are as follows. In Section \ref{sec2}, we present the exact definition of LSC and list some necessary notations. In Section \ref{sec3}, we construct the counter-example of LSC as desired.

\section{Sierpinski carpet like fractals}\label{sec2}
We will consider carpet-like fractals, which are generated with a similar iterated function system (IFS for short) as SC, living symmetrically in a square, with more flexible locations of the cells (copies of the fractal under the composition of mappings in the IFS) as USC, but allowing the contraction ratios of the IFS to be distinct.

Let $\square$ be a unit square in $\mathbb{R}^2$. We let
\[q_1=(0,0),\quad q_2=(1,0),\quad q_3=(1,1),\quad q_4=(0,1)\]
be the four vertices of $\square$, where we write $x=(x_1,x_2)$ for a point in $\mathbb{R}^2$. In addition, we will write
\[\overline{x,y}=\big\{(1-t)x+ty:0\leq t\leq 1\big\},\]
for the line segment connecting points $x,y$ in $\mathbb{R}^2$.

For convenience, we denote the group of self-isometries on $\square$ by
\[\msG=\big\{\Gamma_v,\Gamma_h,\Gamma_{d_1},\Gamma_{d_2},id,\Gamma_{r_1},\Gamma_{r_2},\Gamma_{r_3}\big\},\]
where $\Gamma_v,\Gamma_h,\Gamma_{d_1},\Gamma_{d_2}$ are \textit{reflections} ($v$ for vertical, $h$ for horizontal, $d_1,d_2$ for two diagonals),
\begin{equation}\label{eqn21}
\begin{aligned}
	&\Gamma_v(x_1,x_2)=(x_1,1-x_2),\qquad \Gamma_h(x_1,x_2)=(1-x_1,x_2),\\
	&\Gamma_{d_1}(x_1,x_2)=(x_2,x_1),\qquad\Gamma_{d_2}(x_1,x_2)=(1-x_2,1-x_1),
\end{aligned}
\end{equation}
 for $x=(x_1,x_2)\in \square$; $id$ is the identity mapping; and $\Gamma_{r_1},\Gamma_{r_2},\Gamma_{r_3}$ are \textit{rotations},
\begin{equation}\label{eqn22}
	\Gamma_{r_1}(x_1,x_2)=(1-x_2,x_1),\quad \Gamma_{r_2}=(\Gamma_{r_1})^2,\quad \Gamma_{r_3}=(\Gamma_{r_1})^3,
\end{equation}
around the center of $\square$ counter-clockwisely with angle $\frac{j\pi}{2}$, $j=1,2,3$.

We will always require all the structures (eg: the Dirichlet forms, the diffusions, the measures, etc.) under consideration are $\msG$-symmetric without explicit mention.

\begin{definition}[Sierpinski carpet like fractals (LSC)]\label{def21}\quad
	
Let $k\geq 3$, $N\geq 4(k-1)$, and $\{\rho_i\}_{i=1}^N$ be a collection of positive numbers with $\sum_{i=1}^N\rho_i^2<1$. Let $\{F_i\}_{1\leq i\leq N}$ be a collection of similarities with the form $F_ix=\rho_ix+c_i$ for some $c_i\in \mathbb{R}^2$. Assume the following holds:

\noindent\emph{(Non-overlapping).} $F_i(\square)\cap F_{j}(\square)$ is either a line segment, or a point, or empty, $i\neq j$;

\noindent\emph{(Connectivity). } $\bigcup_{i=1}^N F_i(\square)$ is connected;

\noindent\emph{(Symmetry).} $\Gamma\big(\bigcup_{i=1}^N F_i(\square)\big)=\bigcup_{i=1}^N F_i(\square)$ for any $\Gamma\in \msG$;

\noindent\emph{(Boundary included).} $\overline{q_1,q_2}\subset\bigcup_{i=1}^N F_i(\square)\subset\square$.

We call the unique compact subset $K\subset \square$ satisfying
\[K=\bigcup_{i=1}^N F_iK\]
a {\em Sierpinski carpet like  fractal (LSC)}. 
\end{definition}

Throughout the paper, we always let $d$ be the Euclidean metric on $K$ induced from $\mathbb{R}^2$. Note that by the non-overlapping condition, the invariant set $K$ satisfies the open set condition, i.e. $\bigcup_{i=1}^N F_i(\square^\circ)\subset \square^\circ$ and $F_i(\square^\circ)\cap F_j(\square^\circ)=\emptyset$ for each pair $i\neq j$, where $\square^\circ$ denotes the interior of $\square$. It follows that the Hausdorff dimension of $K$, denoted as $d_H$, is the unique solution of $\sum_{i=1}^N\rho_i^{\alpha}=1$. In addition, the $d_H$-dimensional Hausdorff measure of $K$ is positive and finite. We will always let $\mu$ be the normalized $d_H$-dimensional Hausdorff measure on $K$. That is, $\mu(F_i K)=\rho_i^{d_H}$ for each $i$, and $\mu(F_w K)=\rho_{w_1}^{d_H}\cdots \rho_{w_m}^{d_H}$ for each $w=w_1\cdots w_m\in W_m:=\{1,\cdots,N\}^m$ with $m\geq 0$. 

The condition $k\geq 3$ is to avoid trivial set by the symmetry condition. The condition $N\geq 4(k-1)$ is a requirement of the boundary included condition. 
 The condition $\sum_{i=1}^N\rho_i^2<1$ ensures that $d_H<2$, so that we are dealing with a non-trivial planar self-similar set. 

Note that when $k=3$, $N=8$ and all $\rho_i=\frac 13$, $K$ is the standard SC. Comparing LSC with USC introduced in \cite{CQ}, the main difference is that, the contraction ratios of the IFS are kept to be the same in the later. 
\vspace{0.2cm}

Our naive question is to ask whether there always exist natural `nice' diffusion processes on the more flexible LSC. It was proved in \cite{BB2,BB3,CQ} that the self-similar diffusions on SC, GSC and USC always enjoy a sub-Gaussian heat kernel estimate,
\begin{equation}\label{eqn23}
  \begin{aligned}
\frac{c_1}{t^{d_H/\beta}}\exp({-c_2(\frac{d(x,y)^\beta}{t})^{\frac{1}{\beta-1}}})&\leq p(t,x,y)\\ &\leq \frac{c_3}{t^{d_H/\beta}}\exp({-c_4(\frac{d(x,y)^\beta}{t})^{\frac{1}{\beta-1}}}), \quad \forall 0<t\leq 1,\forall  x,y\in K,
\end{aligned}
\end{equation}
where $p(t,x,y)$ is the \textit{heat kernel} (also called \textit{transition density}), $\beta=-\frac{\log \eta }{\log k}+d_H$ is the \textit{walk dimension}, $0<\eta<1$ is the common renormalization factor of the self-similar diffusion, $c_1-c_4$ are positive constants. To be precise, we need to find certain self-similar Dirichlet forms on LSC.

\begin{definition}\label{def22}
Let $(\mcE,\mcF)$ be a local regular conservative irreducible self-similar Dirichlet form  on an LSC, denoted as $K$. We call $(\mcE,\mcF)$ a standard self-similar Dirichlet form if there exists $\theta\in \mathbb{R}$ such that  
\begin{equation}\label{eqn24}
	\mcE(f)=\sum_{i=1}^N \rho_i^{-\theta}\mcE(f\circ F_i),\qquad\forall f\in \mcF,
\end{equation}
and 
\begin{equation}\label{eqn25}
	f\circ F_i \in \mcF\cap C(K),\forall 1\leq i\leq N\Rightarrow f\in \mcF\cap C(K).
\end{equation}
We abbreviate such Dirichlet forms to SsDF in the later context.
\end{definition}

One may choose other renormalization factors to replace (\ref{eqn24}) to define other self-similar Dirichlet forms. However, if the heat kernel estimate (\ref{eqn23}) is satisfied, then the renormalization factor can only be of the form $\rho_i^{-\theta}$. In fact, for the low dimensional setting $d_H<\beta$, it is well known that (\ref{eqn23}) implies that 
\[c_5 d(x,y)^\theta\leq R(x,y)\leq c_{6} d(x,y)^\theta, \quad \forall x,y \in K,\] 
for $\theta=\beta-d_H$ and some positive constants $c_5, c_6$. See \cite{BCK} for a proof on the discrete setting.

\section{The counter-example}\label{sec3}
In the following, we construct an LSC fractal $K$, associated with an IFS $\{F_i\}$ consisting of $104$ contraction similarities. Let $a=\sqrt{\frac{7}{24}}-\frac{1}{2}\approx 0.0400617$, which is the positive solution of the equation 
\[6(a^2+a)=\frac{1}{4}.\]
First, we define $F_i, 1\leq i\leq 13$ (see Figure \ref{fig2} for an illustration of these mappings),

\[\begin{cases}
	F_{2j+1}(x)=ax+\frac{j}{24},  &\text{ for }0\leq j\leq 5,\\
	F_{2j+2}(x)=a^2x+a+\frac{j}{24},&\text{ for }0\leq j\leq 5,\\
	F_{13}(x)=\frac{1}{4}x+\frac{1}{4}. 
\end{cases}
\]
Next, we use symmetry to extend the above $F_i$ to $1\leq i\leq 100$: let 
\[F_{i}(x)=\Gamma_h\circ F_{27-i}\circ \Gamma_h(x),\quad 14\leq i\leq 26,\]
where $\Gamma_h$ is the horizontal reflection in (\ref{eqn21}); let 
\[F_{i+25j}(x)=\Gamma_{r_j}\circ F_{i}\circ \Gamma_{r_{4-j}}(x),\qquad 1\leq i\leq 25,1\leq j\leq 3,\]
where $\Gamma_{r_j}$'s are the rotations in (\ref{eqn22}). Finally, we define the last $4$ mappings,
\[\begin{aligned}
F_{101}(x)=\frac{1}{4}x+(\frac{1}{4},\frac{1}{4}),\quad F_{102}(x)=\frac{1}{4}x+(\frac{1}{2},\frac{1}{4}),\\
F_{103}(x)=\frac{1}{4}x+(\frac{1}{2},\frac{1}{2}),\quad F_{104}(x)=\frac{1}{4}x+(\frac{1}{4},\frac{1}{2}).
\end{aligned}\]
Let $K$ be the unique compact subset of $\mathbb{R}^2$ satisfying
\[K=\bigcup_{i=1}^{104} F_iK.\]
See Figure \ref{fig3} for an illustration of the IFS and a sketch of $K$. 

\begin{figure}[htp]
	\includegraphics[width=8cm]{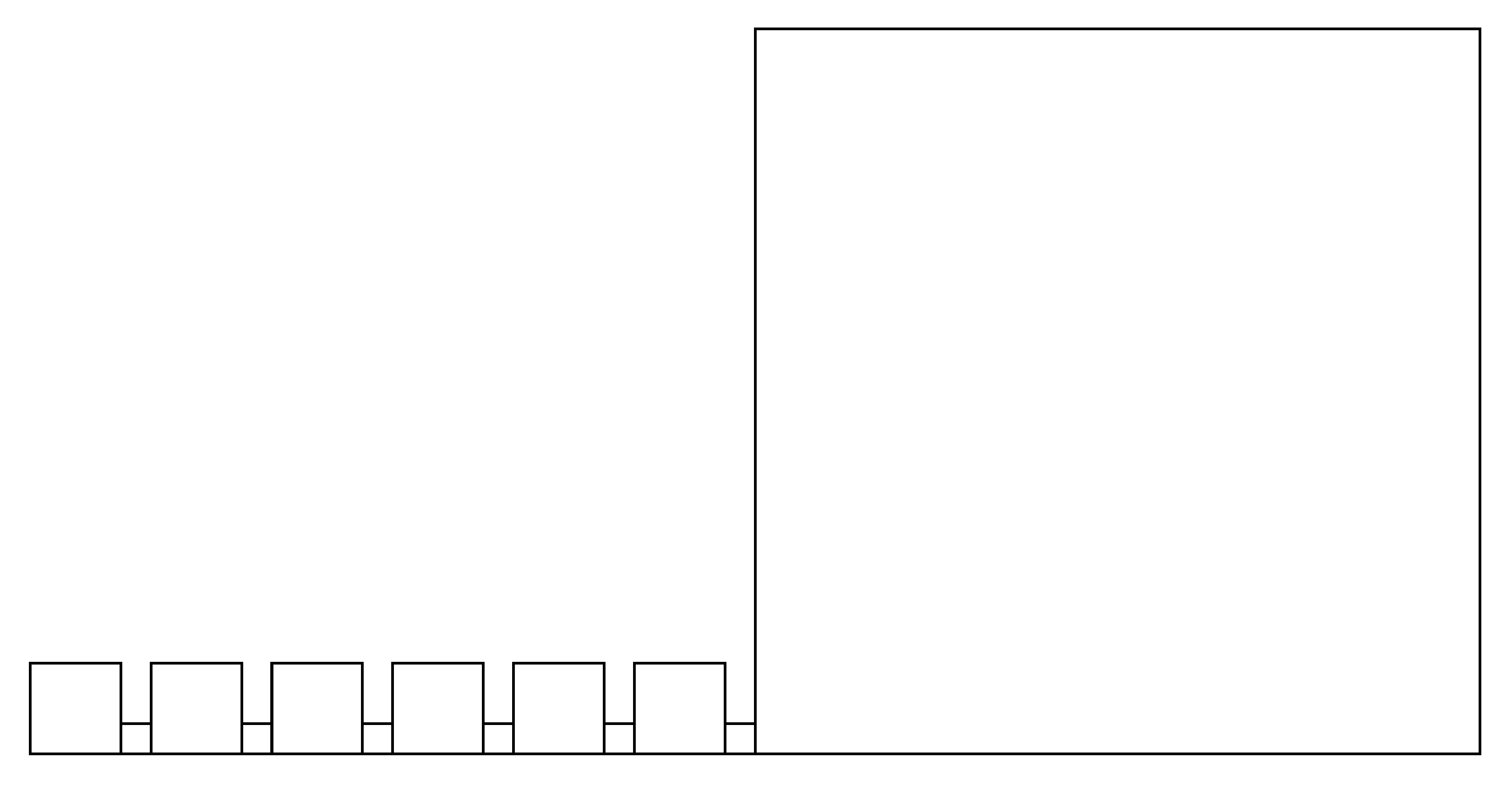}
	\caption{A sketch of mappings $F_i$, $1\leq i\leq 13$ (not using the true value of $a$, which will make $a^2$ to be too small to be distinguished).
}
	\label{fig2}
\end{figure}

\begin{figure}[htp]
	\includegraphics[width=4.9cm]{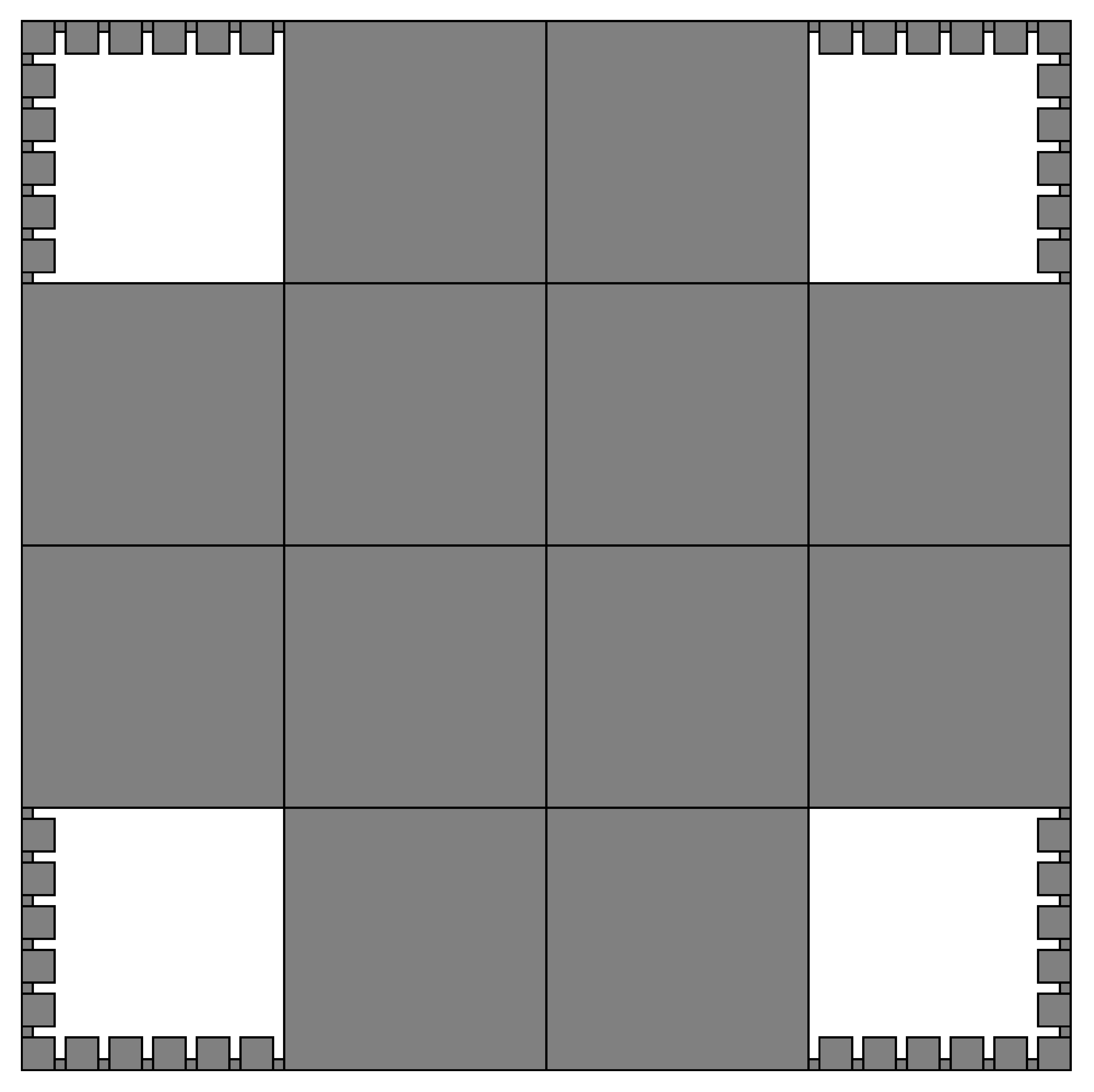}\hspace{1cm}
	\includegraphics[width=4.9cm]{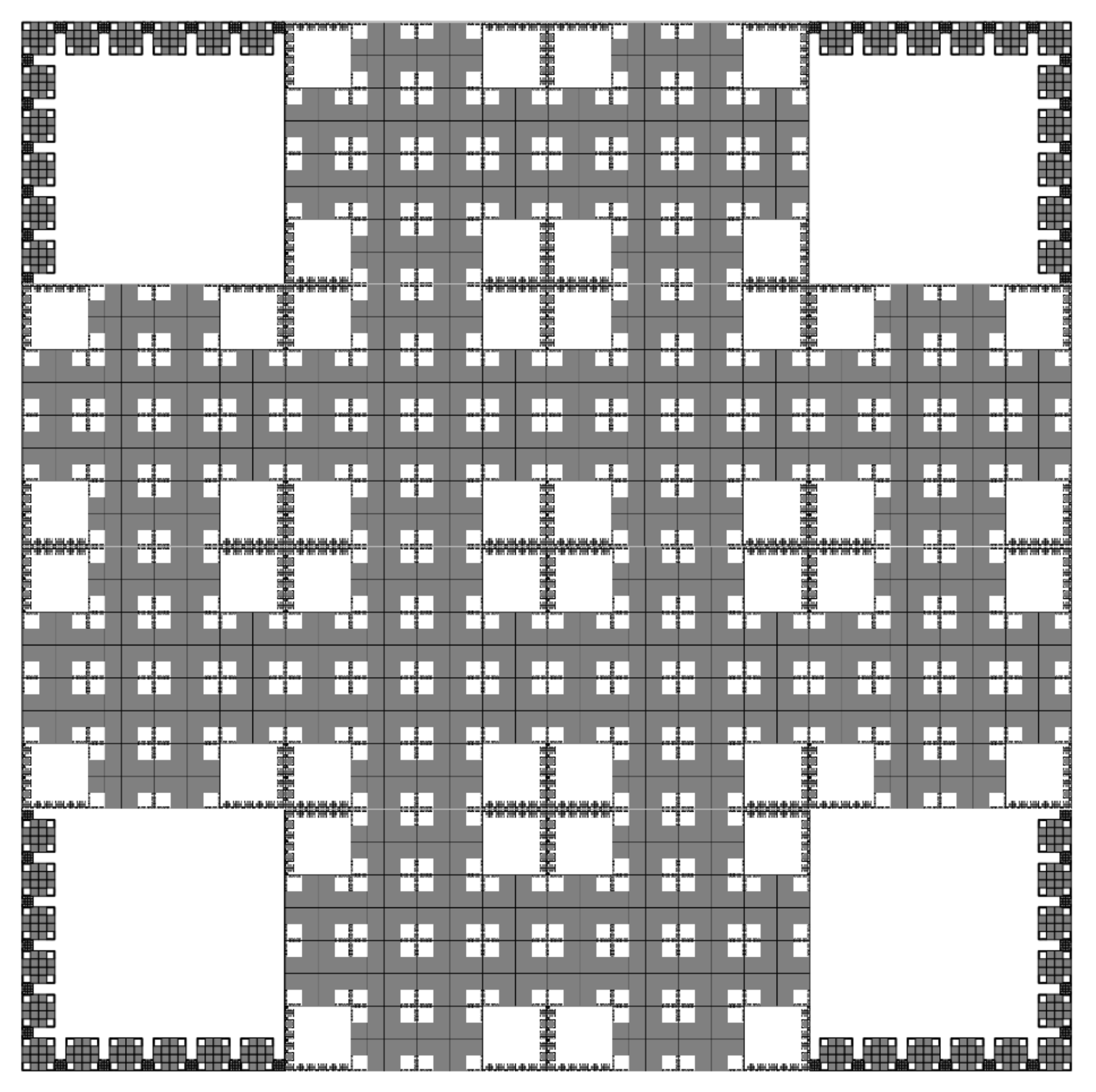}
	\caption{The IFS and a sketch of $K$.}
	\label{fig3}
\end{figure}

In the rest of this section, we will show the following result.

\begin{theorem}\label{thm31}
There is no SsDF on $K$. 
\end{theorem}

We will prove Theorem \ref{thm31} by contradiction. We assume that there is an SsDF on $K$ satisfying  (\ref{eqn24}) and (\ref{eqn25})  for some $\theta\in \mathbb{R}$. Then we will obtain an upper bound estimate $\theta\leq \frac{1}{2}$ and a lower bound estimate $\theta>\frac{1}{2}$, which is of course impossible. Theorem \ref{thm31} will then follow immediately.  

Before proving Theorem \ref{thm31}, we present a few more concepts about capacity. Readers can find details in \cite{CF,FOT}. 

\begin{definition}\label{def32}
	Let $X$ be a compact set, and $\mu$ be a Radon measure on $X$ with full support. Let $(\mcE,\mcF)$ be a regular conservative irreducible Dirichlet form on $L^2(X,\mu)$. We write \[\mcE_1(f)=\mcE(f)+\|f\|^2_{L^2(X,\mu)},\]
	for any $f\in \mcF$.
	
	(a). For each open set $U\subset X$, we define the capacity of $U$ by
	\[Cap(U)=\inf\big\{\mcE_1(f):f|_U\geq 1 \text{ a.e., }f\in \mcF\big\}.\]
 
    For a general Borel set $B\subset X$, define the capacity of $B$ by $Cap(B)=\inf_{U\supset B}Cap(U)$, where the infimum is taken over all open sets $U$ containing $B$. 
    
    (b). For $A\subset X$, let $l(A)$ denote the collection of all real functions on $A$. A function $f\in l(A)$ is called \textit{quasi-continuous} if  for any $\varepsilon>0$, there is open set $U\subset X$ such that $Cap(U)<\varepsilon$ and {$f|_{A\setminus U}\in C(A\setminus U)$}. 
    
    For each $f\in \mcF$, we write $\tilde{f}$ a quasi-continuous modification of $f$ (whose existence is guaranteed by Theorem 2.1.3 in \cite{FOT}). In addition, write  \[\tilde{\mcF}=\{\tilde{f}:f\in \mcF\}.\]
    We say $\tilde{f}=\tilde{g}$ if $\{x:f(x)\neq g(x)\}$ is of zero capacity.
    
    (c). Let $A,B$ be two disjoint closed subsets of $X$ with positive capacity, we denote 
    \[R_{\mcE}(A,B):=\sup\{\frac{1}{\mcE(\tilde{f})}:\tilde{f}\in \tilde{\mcF},\tilde{f}|_A=0,\tilde{f}|_{B}=1\}.\]
    Write $R(A,B)=R_{\mcE}(A,B)$ for short if no confusion caused.
\end{definition}

\begin{lemma}\label{lemma33}
	Let $X$, $\mu$ and $(\mcE,\mcF)$ be as defined in Definition \ref{def32}. Let $A,B$ be two disjoint closed subsets of $X$ with positive capacity. 
	Then, $0<R(A,B)<\infty$. In addition, for any $\varepsilon>0$, there is $f\in \mcF\cap C(X)$ such that $f|_A=0,f|_B=1$ and $\mcE(f)\leq \varepsilon+R^{-1}(A,B)$.
\end{lemma}
\begin{proof}
	First, we show that $R(A,B)<\infty$. Assume that $R(A,B)=\infty$, then there is a sequence $\tilde{f}_n\in \tilde{\mathcal F}$ such that $\tilde{f}_n|_A=0$, $\tilde{f}_n|_B=1$ and $\lim\limits_{n\to\infty}\mcE(\tilde{f}_n)=0$. In addition, we can assume that $0\leq\tilde{f}_n\leq 1$ by the Markov property of $(\mcE,\mcF)$. By taking the limit of the Cesàro mean of a subsequence of $\tilde{f}_n$ (see Theorem A.4.1 of \cite{CF}), one can find some $\tilde{f}\in \tilde{\mcF}$ such that $\mcE(\tilde{f})=0$,  $\tilde{f}|_A=0$ and $\tilde{f}|_B=1$. This contradicts the assumption that $(\mcE,\mcF)$ is irreducible conservative (see Theorem 4.7.1 (iii)  of \cite{FOT}, noticing that the form is recurrent automatically). 
	
	Next we show that $R^{-1}(A,B)=\inf\{\mcE(f):f\in \mcF\cap C(X),f|_A=0,f|_{B}=1\}$. Choose a sequence $f_n\in \mcF\cap C(X)$ such that $f_n|_A=0$, $f_n|_B=1$ and $\lim\limits_{n\to\infty}\mcE(f_n)=\inf\{\mcE(f):f\in \mcF\cap C(X),f|_A=0,f|_{B}=1\}$.  By choosing Ces\`{a}ro mean of $f_n$, we have a limit $f\in \mcF$ in the $\mcE_1$-norm sense. It is direct to check that $\mcE(f+v)\geq \mcE(f)$ for any $v\in \mcF\cap C(X)$ satisfying $v|_{A\cup B}=0$. By using Lemma 2.3.4 in \cite{FOT}, we can see that $\mcE(f)=R^{-1}(A,B)=\inf\{\mcE(\tilde{f}):\tilde{f}\in \tilde{\mcF},\tilde{f}|_A=0,\tilde{f}|_{B}=1\}$. This implies the second statement. 
\end{proof}

\begin{lemma}\label{lemma34}
	Let $(\mcE,\mcF)$ be an SsDF on $K$. Let $L_i=\overline{q_i,q_{i+1}}$ for $i=1,2,3,4$ with cyclic notation $q_5=q_1$. Then, $Cap(L_i)>0$ for all $i=1,2,3,4$.
\end{lemma}
\begin{proof}
	For convenience, we write $\partial_0K=\bigcup_{i=1}^{4}L_i$, and write $\partial_1K=\bigcup_{i=1}^{104}F_i(\partial_0 K)$.
	
	Let $(\Omega,\mathcal{M},X_t,\mathbb{P}_x)$ be the Hunt process associated with $(\mcE,\mcF)$ and $L^2(K,\mu)$, and write $\dot{\sigma}_A=\inf\{t\geq 0: X_t\in A\}$ for a subset $A\subset K$. Since $(\mcE,\mcF)$ is irreducible and local, $\mathbb{P}_x(\dot{\sigma}_{K\setminus F_1K}<\infty)=\mathbb{P}_x(\dot{\sigma}_{F_1(\partial_0 K)}<\infty)>0$ for quasi every $x\in F_1K$ by Theorem 4.7.1 and Theorem 4.5.1 in \cite{FOT}. Thus, $Cap\big(F_1(\partial_0K)\big)>0$ by Theorem 4.2.1 in \cite{FOT}.
	
	Now we show that $Cap(\partial_0 K)>0$. With the previous paragraph, it suffices to show that
	\[Cap(\partial_0 K)\geq c\cdot Cap(\partial_1K),\]
	for some $c>0$. Since $\partial_0 K$ is compact, by Lemma 2.2.7 of \cite{FOT}, there is $f\in \mcF\cap C(K)$ such that $f|_{\partial_0K}=1$ and $\mcE_1(f)\leq Cap(\partial_0K)+\varepsilon$ for any small $\varepsilon>0$. Define $g\in C(K)$ by $g\circ F_i(x)=f(x),\forall 1\leq i\leq N$. Then, by (\ref{eqn25}), we see that $g\in \mcF$ and 
	\[Cap(\partial_1 K)\leq \mcE_1(g)=\sum_{i=1}^N\rho_i^{-\theta}\mcE(f)+\|f\|^2_{L^2(K,\mu)}\leq c_1\big(Cap(\partial_0 K)+\varepsilon\big),\]
	for some $c_1>0$ depending on $\theta$. The desired estimate follows immediately noticing that $\varepsilon$ is arbitrary.
	
	By symmetry, for each $i=1,2,3,4$, $Cap(L_i)\geq\frac{Cap(\partial_0 K)}{4}>0$.
\end{proof}

Now, we return to the proof of Theorem \ref{thm31}. 
\begin{proof}[Proof of Theorem \ref{thm31}]
Assume there is an SsDF on $K$, which means that there is a local regular irreducible conservative self-similar Dirichlet form $(\mcE,\mcF)$ on $K$ satisfying (\ref{eqn24}) and (\ref{eqn25}) with some $\theta\in\mathbb R$.\vspace{0.2cm}

\noindent\textit{Claim $1$}. $\theta\leq\frac{1}{2}$. \vspace{0.2cm}

\textit{Proof of Claim $1$.} We will consider $R(L_2,L_4)$, where as in Lemma \ref{lemma34}, $L_i=\overline{q_i,q_{i+1}}$. By Lemma \ref{lemma33} and \ref{lemma34}, we know that $R(L_2,L_4)>0$ and we can find $\tilde{h}\in \tilde{\mcF}$ such that $\tilde{h}|_{L_2}=0,\tilde{h}|_{L_4}=1$ and $\mcE(\tilde{h})=R^{-1}(L_2,L_4)$. 

Let $S=\{38,39,88,89,101,102,103,104\}$ and $A_S=\bigcup_{i\in S}F_iK$. Note that $A_S=K\cap[0,1]\times[\frac 14,\frac 34]$, i.e. the union of the $8$ $1$-cells connecting $L_2$ to $L_4$ horizontally. Define $(\mcE_S,\mcF_{S})$ as 
\[
\begin{aligned}
\mcF_S=\big\{f\in L^2(A_S,\mu|_{A_S}):&f\circ F_i\in \mcF,\forall i\in S,\\ &\text{ and }\widetilde{f\circ F_i} |_{L_{i'}}=\widetilde{f\circ F_j}|_{L_{j'}},\forall i,i',j,j' \text{ such that }F_iL_{i'}=F_jL_{j'}\big\},
\end{aligned}\]
and 
\[\mcE_S(f)=\sum_{i\in S}4^\theta\cdot\mcE(f\circ F_i). \]
Clearly, we will have $\mcF|_{A_S}\subset \mcF_S$. Also, $L_2,L_4$ still have positive capacity with respect to $(\mcE_S,\mcF_S)$ on $L^2(A_S,\mu|_{A_S})$. By using symmetry, we can easily see that 
\[R^{-1}_{\mcE_S}(L_2\cap A_S,L_4\cap A_S)=8\cdot 4^{-2}\cdot 4^{\theta}\mcE(\tilde{h}),\] and the minimizer of the energy can be obtained by gluing scaled functions of $\frac{1}{4}\tilde{h}+\frac{k}{4}$, $k=0,1,2,3$ on $F_iK$, $i\in S$. So we have $\mcE(\tilde{h})\geq \mcE_S(\tilde{h}|_{A_S})\geq \frac{1}{2}\cdot 4^{\theta}\mcE(\tilde{h})$. This implies that $\theta\leq\frac{1}{2}$. \vspace{0.2cm}

\noindent\textit{Claim $2$}. $\theta\geq -\frac{\log 5}{\log a}$. \vspace{0.2cm}

\textit{Proof of Claim $2$.} We will consider $R(F_1K,F_{26}K)$. Since $F_1K$ and $F_{26}K$ have positive measures, both sets have positive capacity. For a small fixed $\varepsilon>0$, by Lemma \ref{lemma33}, we can find $f\in \mcF\cap C(K)$ such that $f|_{F_1K}=0,f|_{F_{26}K}=1$ and $\mcE(f)\leq R^{-1}(F_1K,F_{26}K)+\varepsilon$. Now we use scaled copies of $f$ to construct a function $g\in\mcF$ through the following $4$ steps:\vspace{0.2cm}

1. First, we construct $g$ on $\bigcup_{i=1}^{12}F_iK$ as 
\[g(x)=\begin{cases}
	0,&\text{ if }x\in F_1K,\\
	\frac{j-1}{10},&\text{ if }x\in F_{2j}K,\text{ for }1\leq j\leq 6,\\ 
	\frac{1}{10}f\circ F_{2j+1}(x)+\frac{j-1}{10},&\text{ if }x\in F_{2j+1}K,\text{ for }1\leq j\leq 5.
\end{cases}\]
Note that the requirement, that the contraction ratio of $F_{2j}K$, $1\leq j\leq 5$ is the square of its neighboring cells, ensures the continuity of $g$.

2. Next, we extend $g$ by symmetry to $\bigcup_{i=90}^{100}F_iK$. More precisely, let $g(x)=g\circ \Gamma_{d_1}(x)$ for any $x\in \bigcup_{i=90}^{100}F_iK$.

3. Then, we extend $g$ to the left half of $K$ by $\frac{1}{2}$, i.e. let $g(x)=\frac{1}{2}$ for any $x\in \Big(K\cap([0,\frac 12]\times[0,1])\Big)\setminus \Big((\bigcup_{i=1}^{12}F_iK)\cup(\bigcup_{i=90}^{100}F_iK)\Big)$.

4. Finally, we extend $g$ to $K$ so that $g-\frac{1}{2}$ is antisymmetric with respect to $\Gamma_h$, i.e. $g\circ \Gamma_h(x)+g(x)=1$ for any $x\in K$. \vspace{0.2cm}

Now, we can easily see that $g\in \mcF\cap C(K)$, $g|_{F_1K}=0$, $g|_{F_{26}K}=1$ and 
\[\mcE(g)=20\cdot (\frac{1}{10})^2\cdot a^{-\theta}\mcE(f)=\frac{1}{5}a^{-\theta}\mcE(f).\]
Since $\mcE(g)\geq R^{-1}(F_1K,F_{26}K)$, $\mcE(f)\leq R^{-1}(F_1K,F_{26}K)+\varepsilon$ and $\varepsilon$ is arbitrary, we finally get 
$\frac15a^{-\theta}\geq 1$, and thus $\theta\geq -\frac{\log 5}{\log a}$.\vspace{0.2cm}

Finally, we can see there is a contradiction between Claim $1$ and Claim $2$, noticing that $-\frac{\log 5}{\log a}>\frac{1}{2}$. 
\end{proof}

We conclude the paper with two remarks.
\begin{remark}\label{rem35}
Our construction is partially inspired by the celebrated work of Sabot \cite{Sabot} on p.c.f. self-similar sets \cite{ki1,ki2}. If one try to reproduce the arguments in Subsection 4.2 of \cite{CQ} with reasonable modification, one can see that the effective resistances between corner vertices are comparable with the Poincare constants. This observation motivates the authors to construct the example with corner points loosely connected with inner cells, so that $R_n\ll\lambda_n$.  On the other hand, the authors believe that there exists good Dirichlet forms on many other LSC.
\end{remark}

\begin{remark}
Based on a same idea, the construction can be extended to their higher dimensional analogues (where we can assume the boundary faces are included in the fractal). We still only need three different sizes of cells, and the key to the construction is by choosing small size of $1$-cells near the border of each face, and by placing the cells suitably, we can get an lower bound $\theta\geq 3-d-\varepsilon$, where $d$ is the dimension. 

For the $2$-dimensional  case, one can simply suppose that $(\mcE,\mcF)$ is a resistance form \cite{ki2}, since we want the sub-Gaussian heat kernel estimates. However, we still keep the Dirichlet form setting in the paper so the arguments can be extended to higher dimensions.
\end{remark}

\bibliographystyle{amsplain}

\begin{thebibliography}{10}

\bibitem{BB}
M.T. Barlow and R.F. Bass, \emph{The construction of Brownian motion on the Sierpinski carpet}, Ann. Inst.
Henri Poincar\'{e} 25 (1989), no. 3, 225--257.

\bibitem{BB1}
M.T. Barlow and R.F. Bass, \emph{On the resistance of the Sierpiński carpet,} Proc. Roy. Soc. London Ser. A  431 (1990), no. 1882, 345--360.

\bibitem{BB2}
M.T. Barlow and R.F. Bass, \emph{Transition densities for Brownian motion on the Sierpinski carpet,} Probab. Theory Related Fields 91 (1992), 307--330.

\bibitem{BB3}
M.T. Barlow and R.F. Bass, \emph{Brownian motion and harmonic analysis on Sierpinski carpets,} Canad. J. Math. 51 (1999), no. 4, 673--744.

\bibitem{BBKT} M.T. Barlow, R.F. Bass, T. Kumagai and A. Teplyaev, \emph{Uniqueness of Brownian motion on Sierpinski carpets}, J. Eur. Math. Soc. 12 (2010), no. 3, 655--701.

\bibitem{BCK} M.T. Barlow, T. Coulhon and T. Kumagai, \emph{Characterization of sub-Gaussian heat kernel estimates on strongly recurrent graphs. (English summary)}
Comm. Pure Appl. Math. 58 (2005), no. 12, 1642--1677.

\bibitem{CQ} S. Cao and H. Qiu, \emph{Dirichlet forms on unconstrained Sierpinski carpets}, arXiv:2104.01529.

\bibitem{CF} Z. Chen and M. Fukushima, \emph{Symmetric Markov processes, time change, and boundary theory.} London Mathematical Society Monographs Series, 35. Princeton University Press, Princeton, NJ, 2012. xvi+479 pp. 

\bibitem{FOT} M. Fukushima, Y. Oshima and M. Takeda, \emph{Dirichlet forms and symmetric Markov processes.} Second revised and extended edition,  De Gruyter Studies in Mathematics, 19. Walter de Gruyter \& Co., Berlin, 2011.

\bibitem{ki1} J. Kigami, \emph{A harmonic calculus on p.c.f. self-similar sets,} Trans. Amer. Math. Soc. 335 (1993), no. 2, 721--755.

\bibitem{ki2} J. Kigami, \emph{Analysis on Fractals.} Cambridge Tracts in Mathematics, 143. Cambridge University Press, Cambridge, 2001.

\bibitem{KZ}
S. Kusuoka and X.Y. Zhou, \emph{Dirichlet forms on fractals: Poincar\'{e} constant and resistance}, Probab. Theory Related Fields 93 (1992), no. 2, 169--196.

\bibitem{Sabot}
C. Sabot, \emph{Existence and uniqueness of diffusions on finitely ramified self-similar fractals (English, French summary),} Ann. Sci. \'{E}cole Norm. Sup. 30 (1997), no. 5, 605--673.

\end{thebibliography}

\end{document}